\documentclass{amsart}%
\usepackage{cite}
\usepackage[hypertex]{hyperref}
\usepackage{amssymb}
\usepackage{amsmath}
\usepackage{amsthm}
\usepackage{amsfonts}
\usepackage{graphicx}
\usepackage{setspace}
\usepackage{color}
\usepackage{MnSymbol}
\usepackage{mathrsfs}
\usepackage[toc,page]{appendix}%
\usepackage{multirow}

\usepackage{marginnote}

\newcommand{\w}{\omega}

\newcommand{\spn}{\operatorname{span}}
\newcommand{\Null}{\operatorname{Null}}
\newcommand{\Horz}{\mathcal{H}}
\newcommand{\V}{\mathcal{V}}
\newcommand{\LV}{\mathcal{L}^{\mathcal{V}}}
\newcommand{\diverg}{\operatorname{div}}
\newcommand{\grad}{\operatorname{grad}}
\newcommand{\gradH}{\grad_{\Horz}}

\newcommand{\partialx}[1]{\frac{\partial}{\partial x^{#1}}}
\newcommand{\mfm}{\mathfrak{m}}
\newcommand{\mfmi}{\mathfrak{m_i}}

\newtheorem{thm}{Theorem}[section]
\newtheorem{cor}[thm]{Corollary}

\newtheorem{assumption}{Assumption}
\newtheorem{lem}[thm]{Lemma}
\newtheorem{nota}[thm]{Notation}
\newtheorem{prop}[thm]{Proposition}

\newtheorem{as}[thm]{Assumption}

\theoremstyle{definition}
\newtheorem{df}[thm]{Definition}
\newtheorem*{claim}{Claim}

\newtheorem{example}{Example}[section]
\newtheorem{rem}[thm]{Remark}

\numberwithin{equation}{section}
\newtheorem*{acknowledgement}{Acknowledgement}

\begin{document}

\title[Sub-Riemannian Laplacians]{Sub-Laplacians on sub-Riemannian manifolds}
\author[Gordina]{Maria Gordina{$^{\dagger}$}}
\thanks{\footnotemark {$\dagger$} This research was supported in part by NSF Grant DMS-1007496.}
\address{Department of Mathematics\\
University of Connecticut\\
Storrs, CT 06269, USA} \email{maria.gordina@uconn.edu}
\author[Laetsch]{Thomas Laetsch{$^{\dagger}$}}
\address{Department of Mathematics\\
University of Connecticut\\
Storrs, CT 06269, USA} \email{thomas.laetsch@uconn.edu}

\keywords{sub-Riemannian manifold, sub-Laplacian, hypoelliptic operator}

\subjclass{Primary 53C17, 35R01; Secondary 58J35}

\date{\today \ \emph{File:\jobname{.tex}}}

\begin{abstract} We consider different sub-Laplacians on a sub-Riemannian manifold $M$. Namely, we compare different natural choices for such operators, and give conditions under which they coincide. One of these operators is a sub-Laplacian we constructed previously in \cite{GordinaLaetsch2014a}. This operator is canonical with respect to the horizontal Brownian motion, we are able to define the sub-Laplacian without some a priori choice of measure. The other operator is $\diverg^{\w} \gradH$ for some volume form $\w$ on $M$. We illustrate our results by examples of three Lie groups equipped with a sub-Riemannian structure: $\operatorname{SU}\left( 2 \right)$, the Heisenberg group and the affine group.
\end{abstract}

\maketitle
\tableofcontents

\section{Introduction} In the present paper we study operators on sub-Riemannian manifolds which can be considered as geometrically natural analogues of the Laplace-Beltrami operators in the Riemannian setting. Some of the fundamental difficulties include absence of a canonical measure such as the Riemannian volume measure, and therefore lack of a naturally defined divergence of a vector field, and degeneracy of the metric that prevents us from using local formula for such an operator.

Sub-Riemannian geometry appears in many areas, for example, describing constrained systems in mechanics, or as limiting cases of Riemannian geometries. Roughly speaking, a sub-Riemannian manifold is a smooth manifold $M$ endowed with a bracket-generating (completely non-integrable) sub-bundle $\Horz$ of the tangent bundle $TM$ and a smooth fiberwise inner product on $\Horz$; the sub-bundle $\Horz$ is called the horizontal distribution. The degeneracy of operators defined only in terms of horizontal vector fields (smooth sections of $\Horz$) make sub-Riemannian manifolds natural settings to study sub-elliptic operators which are, in fact, hypoelliptic by an application of H\"{o}rmander's theorem  \cite{Hormander1967a} with the bracket generating assumption.  A more detailed review of these structures can be found in Section \ref{Sec: sub-Riem}.

Our approach is to compare two operators on a sub-Riemannian manifold $M$ that can be thought of as geometrically canonical to the sub-Riemannian structure we have on $M$. One of these operators, $\LV$,  is a sub-Laplacian we constructed previously in \cite{GordinaLaetsch2014a}. The advantage of this construction is that while it is canonical with respect to the horizontal Brownian motion, we are able to define the sub-Laplacian without some a priori choice of measure. Another operator we consider is $\diverg^{\w} \gradH$ for some volume form $\w$ on $M$ in Section \ref{sec: divgrad} which certainly depends on the form $\w$. This comparison culminates in Theorem \ref{thm.LV=divgradH} which gives necessary and sufficient conditions for these two operators to be equal.

In conclusion we want mention a number of related results. First of all, since Lie groups provide a number of meaningful examples, it is natural that there were several results in that setting, in particular, \cite{AgrachevBoscainGauthierRossi2009}. Their approach is to choose a reference measure out of several candidates such as Hausdorff or Popp's measure, which happens to be  scalar multiples of a Haar measure on a Lie group $G$. Popp's measure is attractive since local isometries are volume preserving, which uniquely identifies Popp's measure when the group of isometries of $G$ acts transitively on $G$. From this we also deduce that on Lie groups equipped with a left-invariant sub-Riemannian metric, Popp's measure is proportional to the left Haar measure. For a nice exposition on Popp's measure and the resulting sub-Laplacian, we refer the reader to \cite{BarilariRizzi2013}. It is not uncommon, however, to consider a left-invariant structure on $G$ while endowing $G$ with a right Haar measure.  To see that the choice of the left-invariant structure on $G$ with the right Haar measure is natural for study of sub-elliptic heat kernels we refer to \cite{DriverGrossSaloff-Coste2009a}. We refrain from making a single  choice of measure and illustrate our main results by looking at three examples in Section \ref{s.examples}. We consider our construction as a starting point of further studies of such sub-Laplacians including the corresponding heat kernel estimates, and connecting it to \cite{BaudoinGarofalo2011, BakryBaudoinBonnefont2009, BaudoinBonnefont2012} which will give rise to a number of functional inequalities.

\section{Sub-Riemannian Manifolds}\label{Sec: sub-Riem}
We start by recalling the standard definition of a sub-Riemannian manifold.
\begin{df} Let $M$ be a $d$-dimensional, connected, smooth manifold with tangent and cotangent bundles $TM$ and $T^*M$ respectively. Suppose that $\Horz \subset TM$ is an $m$-dimensional smooth sub-bundle such that the sections of $\Horz$ satisfy  \emph{H\"{o}rmander's condition} (the \emph{bracket generating condition}) formulated in Assumption \ref{HCassump}.
Suppose further that on each fiber of $\Horz$ there is an inner product $\langle \cdot, \cdot \rangle$ which varies smoothly between fibers.  In this case, the triple $(M, \Horz, \langle \cdot, \cdot \rangle)$ is called a \emph{sub-Riemannian manifold} of rank $m$, $\Horz$ is called the \emph{horizontal distribution}, and $\langle \cdot, \cdot \rangle$ is called the \emph{sub-Riemannian metric}. The vectors (resp.~vector fields)  $X \in \Horz$ are called \emph{horizontal vectors} (resp.~horizontal vector fields), and curves $\sigma$ in $M$ whose tangent vectors are horizontal, are called \emph{horizontal curves}.
\end{df}
Having been given $M$ and $\Horz$, the discussion in Section \ref{s.linear} gives us an alternative equivalent approach to the sub-Riemannian structure defined by a sub-Riemannian metric. Indeed, we could have alternatively introduced the symmetric, positive semi-definite \emph{sub-Riemannian bundle homomorphism} $\beta : T^*M \to TM$ such that $\beta(T^*M) = \mathcal{H}$ which is in unique correspondence with the sub-Riemannian metric through the equality $\langle \beta(p), X \rangle = p(X)$ which holds for all $1$-forms $p$ and horizontal vector fields $X$.

\begin{nota} \label{nota.subR1}
We will use $\{X_1, ..., X_m\}$ to denote a (local) \emph{horizontal frame}, that is, a set of vector fields which form a (local) fiberwise basis for $\Horz$. Further, we let $(x^1, ..., x^d)$ represent a (local) chart with corresponding tangent frame $\big\{ \partialx{1}, ..., \partialx{d} \big\}$ and dual frame $\{ dx^1, ..., dx^d\}$.  Finally, we define the smooth maps $\beta^{ij} = \langle \beta(dx^i), \beta(dx^j) \rangle = dx^i\left( \beta(dx^j)\right)$.
\end{nota}

\begin{rem}
If $\Horz = TM$, then $(M, \Horz, \langle \cdot, \cdot \rangle)$ is a Riemannian manifold. In this case, $\beta^{ij}$ is the familiar ``index raising operator'' $g^{ij}$ defined as the inverse of the metric $g_{ij} = \langle \partial/\partial x^i, \partial/\partial x^j \rangle$.
\end{rem}

\subsection{H\"{o}rmander's condition and its consequences}

\begin{assumption}(H\"{o}rmander's condition)\label{HCassump}
We will say  that $\Horz$ satisfies \emph{H\"{o}rmander's (bracket generating) condition} if horizontal vector fields with their Lie brackets span the tangent space $T_{p}M$ at every point $p \in M$.
\end{assumption}
As we remark below in Definition \ref{defi.subLaplace}, H\"{o}rmander's condition guarantees that every sub-Laplacian is hypoelliptic. In addition, H\"{o}rmander's condition has significant topological consequences. We recall the important Chow-Rashevski Theorem below, for more details we refer the reader to \cite{MontgomeryBook2002}. To this end, we define the \emph{Carnot-Caratheodory metric} $d_{CC}$ on $M$ by
\begin{align}\label{e.3.1}
& d_{CC}(x,y) =
\\
& \inf\left\{ \left( \int_0^1 |\sigma'(t)|^2 \, dt\right)^{2}~ \text{ where } \sigma(0) = x, \sigma(1) = y, \sigma \text{ is a horizontal path}\right\}, \notag
\end{align}
where as usual, $\inf(\emptyset):= \infty$. It is not immediately obvious that given any two points $x, y \in M$, that $d_{CC}(x,y) < \infty$; indeed, it would not be impossible to believe that perhaps there is no horizontal curve connecting $x$ and $y$. Yet, remarkably, H\"{o}rmander's condition is sufficient to ensure that any two points are connected by (a finite length) horizontal curve. In fact, even more is true.

\begin{thm}[Chow-Rashevski]
Suppose $\Horz$ satisfies H\"{o}rmander's condition in a neighborhood of every point in $M$. Then for any two points $x,y \in M$, $d_{CC}(x,y) < \infty$. Moreover, the topology on $M$ defined by $d_{CC}$ agrees with the original manifold topology of $M$.
\end{thm}

\subsection{Hamilton-Jacobi Equations}
For physical reasons, we will commonly refer to $M$ as a \emph{configuration space}, vectors $X \in TM$ as \emph{velocity vectors}, and covectors $p \in T^*M$ as \emph{momentum vectors}. The Hamiltonian $H : T^*M \to \mathbb{R}$ is the (kinetic energy) map defined by
\begin{equation}\label{eqn.Hamiltonian}
H(x, p) = \frac{1}{2} \left. \langle \beta(p), \beta(p) \rangle\right|_x = \frac{1}{2} p_{i}p_{j} \beta^{ij}(x),
\end{equation}
where the second equality is a local expression with $p = \sum\limits_{i=1}^d p_i \, dx^i|_x$. A curve $p(t) = (x(t), p(t))$  in $T^{\ast}M$ is said to satisfy  the \emph{Hamilton-Jacobi equations} when

\begin{align}
&\dot{x}^i = \frac{\partial H}{\partial p_i} (x(t),p(t)), \label{eqn.HamiltonsEq}
\\
&\dot{p}_i = -\frac{\partial H}{\partial x^i} (x(t),p(t)). \notag
\end{align}

Note that with a starting position $x(0) = x \in M$ and \emph{momentum} $p(0) = p \in T^{\ast}_{x}M$, we can uniquely solve \eqref{eqn.HamiltonsEq} for some interval of time. The same can not be said if we are given an initial position $x(0) = x$ and horizontal \emph{velocity} $\dot{x}(0) = X \in H$; this is an artifact of the degeneracy of $\beta$, since $\beta^{-1}(X)$ is multi-valued, and there is no a priori canonical choice of which momentum $p \in \beta^{-1}(X)$ to choose.

Our final note on solutions to the Hamilton-Jacobi equations in the sub-Riemannian setting deals with completeness, see \cite[Theorem 7.1]{Strichartz1986a}.

\begin{thm}[Hopf-Rinow Theorem for sub-Riemannian manfiolds]\label{thm.HopfRinow}
 If $M$ is complete as a metric space with respect to $d_{CC}$, then for ever $x \in M$ and $p \in T_x^*M$, the solution of \eqref{eqn.HamiltonsEq} with initial conditions $x(0) = x$ and $p(0) = p$ is  defined for all times $t \geqslant 0$.
\end{thm}

\section{Sub-Riemannian analogues of the Laplace-Beltrami operator}
We start by recalling how the Laplace-Beltrami operator $\Delta_{LB}$ on an oriented $d$-dimensional Riemannian manifold $(M,g)$ is usually defined. First one constructs the Riemannian volume
\[
\w := \sqrt{|g|}\, dx^1 \wedge \cdots \wedge dx^d,
\]
and the respective divergence of vector fields
\[
\diverg^{\w}(X) = \sum\limits_{k=1}^d \frac{1}{\sqrt{|g|} }\frac{\partial}{\partial x^k} \sqrt{|g|}\, X^k.
\]
Here, as usual, $|g|$ is the determinant of the metric. From this the Laplace-Beltrami operator is defined as $\Delta_{LB} = \diverg^{\w} \grad$, which locally is given by
\begin{equation}\label{eqn.LB}
\Delta_{LB} = \sum\limits_{i,j=1}^d \left\{ g^{ij} \, \frac{\partial^2}{\partial x^i \partial x^j} - \sum\limits_{k=1}^d \Gamma^{ijk} g_{ij} \frac{\partial}{\partial x^k}\right\},
\end{equation}
where
\begin{equation}\label{eqn.raisedChrisR}
 \Gamma^{ijk} := - \frac{1}{2} \sum\limits_{l=1}^d\left\{ g^{il} \frac{\partial g^{jk}}{\partial x^l} + g^{j l } \frac{\partial g^{i k }}{\partial x^l} - g^{l k} \frac{\partial g^{ij}}{\partial x^l} \right\}
 \end{equation}
 are the raised Christoffel symbols.

There are multiple problems when we try to use this approach in the sub-Riemannian setting to define a canonical analogue of the Laplace-Beltrami operator. Without a Riemannian metric, the corresponding Riemannian volume form and hence the divergence is left undefined since the $\sqrt{|g|}$ term has no canonical interpretation in general. We could extend the sub-Riemannian metric to a Riemannian metric and use the extension to give meaning to $\sqrt{|g|}$, but generally no one extension seems to stand out as the canonical choice. Moreover, if we just apply \eqref{eqn.LB} with some metric extension $g$,  we would simply be considering the Laplace-Beltrami operator associated to the Riemannian manifold $(M,g)$, rather than to the original sub-Riemannian structure.

While perhaps there is no general best choice for an analogue of the Laplace-Beltrami operator, there are several candidates which merit considering. The remainder of this section will be dedicated to exploring common features of  such operators.


\subsection{Sub-Laplacians}\label{sec: Sub-Laplacians}

\begin{df} \label{defi.subLaplace}
A second order differential operator $\Delta$ defined on $C^{\infty}\left( M \right)$ will be called a \emph{sub-Laplacian} when for every $x \in M$ there is a neighborhood $U$ of $x$ and a collection of smooth vector fields $\{X_0, X_1, ..., X_m\}$ defined on $U$  such that $\{X_1, ..., X_m\}$ are orthonormal with respect to the sub-Riemannian metric and
\[
\Delta = \sum\limits_{k=1}^m X_k^2 + X_0.
\]
\end{df}

By the classical theorem of L.~H\"{o}rmander in \cite[Theorem 1.1]{Hormander1967a} Assumption \ref{HCassump} guarantees that any sub-Laplacian is hypoelliptic.  We now work towards a local coordinate classification of sub-Laplacians, resulting in Corollary \ref{cor.LLclass}. We start with a lemma.
\begin{lem}\label{l.5.2}
Suppose that $p_1, p_2$ are two one-forms  and that $\{X_i\}_{i=1}^m$ is an orthonormal horizontal frame within some neighborhood $U \subset M$. Then  within $U$,
\[
\langle \beta(p_1), \beta(p_2) \rangle = \sum\limits_{k=1}^m  \langle \beta(p_1), X_k \rangle\langle X_k, \beta(p_2) \rangle = \sum\limits_{k=1}^m p_1(X_k)p_2(X_k).
\]
\end{lem}
\begin{proof}
Since $\beta(p_i)$ yields a horizontal vector field ($i = 1, 2$), then within $U$, $\beta(p_i) = \sum_{k=1}^m \langle \beta(p_i), X_k \rangle X_k$. Hence
\[
\langle \beta(p_1), \beta(p_2) \rangle = \Big\langle  \sum_{k=1}^m \langle \beta(p_i), X_k \rangle X_k, \beta(p_2) \Big\rangle  = \sum\limits_{k=1}^m  \langle \beta(p_1), X_k \rangle\langle X_k, \beta(p_2) \rangle.
\]
This proves the first equality; the second equality is shown by defining  $\beta$ by $\langle \beta(p), X \rangle = p(X)$ for any covector $p$ and vector $X$.
\end{proof}

From Lemma \ref{l.5.2} we can conclude the following.
\begin{prop} \label{prop.sumsqbetaij}
Let $\{X_i\}_{i=1}^m$ be a local orthonormal horizontal frame. In local coordinates
\[
X_1^2 + \cdots + X_m^2 = \sum\limits_{i,j=1}^d \beta^{ij} \frac{\partial^2}{\partial x^i \partial x^j} + \text{ first order terms}
\]
As usual, $\beta^{ij} := \langle \beta(dx^i), \beta(dx^j) \rangle$.
\end{prop}
\begin{proof}
Let $k \in \{1,...,m\}$. We have $X_k = \sum\limits_{i=1}^d dx^i(X_k) \frac{\partial}{\partial x^i} = \sum\limits_{i=1}^d \langle \beta(dx^i),X_k\rangle \frac{\partial}{\partial x^i}$, where again the last equality is simply through the definition of $\beta$. Hence
\begin{align*}
\sum_{k=1}^m X_k^2 &= \sum_{k=1}^m \sum\limits_{i,j=1}^d \bigg( \langle \beta(dx^i),X_k\rangle \frac{\partial}{\partial x^i} \bigg) \bigg( \langle X_k, \beta(dx^j)\rangle \frac{\partial}{\partial x^j} \bigg) \\
&= \sum_{k=1}^m \sum\limits_{i,j=1}^d \langle \beta(dx^i), X_k \rangle \langle \beta(dx^j), X_k \rangle \frac{\partial^2}{\partial x^i \partial x^j} + \text{ first order terms}
\end{align*}
Summing over $k$ and using the previous lemma, we get
\begin{align*}
\sum_{k=1}^m X_k^2 &= \sum_{i,j=1}^d \langle \beta(dx^i), \beta(dx^j) \rangle \frac{\partial^2}{\partial x^i \partial x^j} + \text{ first order terms} \\
&= \sum_{i,j=1}^d \beta^{ij} \frac{\partial^2}{\partial x^i \partial x^j} + \text{ first order terms}.
\end{align*}
This concludes the proof.
\end{proof}

We immediately deduce the following.
\begin{cor} \label{cor.LLclass}
 $\Delta$ is a sub-Laplacian if and only if there is a smooth vector field $X_0$ such that locally
\[
\Delta = \sum\limits_{i,j=1}^d \beta^{ij} \frac{\partial^2}{\partial x^i \partial x^j} + X_0.
\]
In particular, the principal symbol of any sub-Laplacian has the form $\beta^{ij} \xi_i \xi_j$.
\end{cor}

\subsection{Lie Groups}
For this section we assume $M=G$ is a Lie group with Lie algebra $\mathfrak{g} = T_e G$. We consider what can be inferred by imposing structure on $\mathcal{H}$ natural to the Lie group.

\begin{as} \label{property.HLI}
For any $v \in \mathcal{H}_e\subset \mathfrak{g}$, the corresponding (unique) left-invariant vector field $X$ defined by $X_x = (L_x)_{\ast} v$ is horizontal.
\end{as}

\begin{as} \label{property.metricLI}
The sub-Riemannian metric $\langle \cdot, \cdot \rangle$ is left-invariant. That is, for any two left-invariant horizontal vector fields $X$ and $Y$, $\langle X, Y \rangle \equiv \langle X_e, Y_e \rangle$.
\end{as}

We immediately get the following Lemma.
\begin{lem}
If Assumption \ref{property.HLI} holds and  $\{v_1, ..., v_m\} \subset \mathcal{H}_e$ is a basis of $\mathcal{H}_e$, then the corresponding left invariant vector fields $\{X_1, ..., X_m\}$ form a (global) horizontal frame. If further Assumption \ref{property.metricLI} holds and $\{v_1, ..., v_m\}$ are orthonormal, then the collection $\{X_1, ..., X_m\}$ is an  orthonormal  horizontal frame.
\end{lem}
Note that the next result does not assume H\"{o}rmander's condition (Assumption \ref{HCassump}).
\begin{thm} \label{thm.leftinvsumsq}
Assume that Assumption \ref{property.HLI} holds. Suppose that $\{v_1, ..., v_m\}$ and $\{r_1, ..., r_m\}$ are two orthonormal bases of $\mathcal{H}_e$ with corresponding left invariant vector fields $\{X_1, ..., X_m\}$ and $\{Y_1, ..., Y_m\}$ respectively. Then, $\sum\limits_{k=1}^m X_k^2 = \sum\limits_{k=1}^m Y_k^2$.
\end{thm}
\begin{proof}
Let $\Theta$ be the $m \times m$ orthogonal matrix with entries ${\theta_i}^j$ such that $v_i = \sum\limits_{j=1}^m {\theta_i}^j r_j$ for each $i = 1, 2, ..., m$. Arguing by the uniqueness of left invariant vector fields, this means that $X_i = \sum\limits_{j=1}^m {\theta_i}^j Y_j$. Symbolically, if ${\bf X} = (X_1, ..., X_m)^{tr}$ and ${\bf Y} = (Y_1, ..., Y_m)^{tr}$, then
\[
{\bf X} = \Theta {\bf Y}.
\]
From this, arguing formally,
\[
\sum\limits_{k=1}^m X_k^2 = {\bf X}^{tr} {\bf X} = \big({\bf Y}^{tr} \Theta^{tr}\big)\big( \Theta {\bf Y} \big) ={\bf Y}^{tr} {\bf Y} = \sum\limits_{k=1}^m Y_k^2,
\]
where the penultimate equality is due to the orthogonality of $\Theta$. In fact, this proof is rigorous upon writing
\[
\sum_{k=1}^m X_k^2 = \sum_{k=1}^m \Big(\sum\limits_{i=1}^m {\theta_k}^i Y_i \Big) \Big(\sum\limits_{j=1}^m {\theta_k}^j Y_j \Big) = \sum_{i,j=1}^m \Big(\sum\limits_{k=1}^m {\theta_k}^i {\theta_k}^j \Big) Y_i Y_j
\]
and realizing that $\sum\limits_{k=1}^m {\theta_k}^i {\theta_k}^j$ is the $ij$th entry of $\Theta^{tr} \Theta = \operatorname{Id}$.
\end{proof}

\begin{example}[A non-example]
In Section \ref{example.heisenberg} below, we introduce the Heisenberg group $\mathbb{H}$ endowed with the the left invariant frame $\{X,Y,Z\}$ defined by $X = \partial_x - \frac{1}{2} y \partial_z,$ $Y = \partial_y + \frac{1}{2} x \partial_z,$ and $Z = \partial_z.$
The horizontal distribution is given by $\mathcal{H} = \operatorname{span}\{X,Y\}$ with sub-Riemannian metric defined so that $\{X,Y\}$ is an orthonormal horizontal frame.

Let us define the new horizontal frame $\{X', Y'\}$ by
\begin{align*}
X' = \cos z \, X - \sin z \, Y \\
Y' = \sin z \, X + \cos z \, Y.
\end{align*}
You will recognize this as a $z$-dependent rotation of the $\{X, Y\}$ frame in $\mathcal{H}$. In particular, $\{X', Y'\}$ is still an orthonormal frame for $\mathcal{H}$ with respect to the sub-Riemannian metric, yet it is not a left-invariant frame. We find
\begin{align*}
X' = \cos z \, \partial_x -\sin z \, \partial_y -  \frac{1}{2}(x \sin z + y \cos z)\, \partial_z
\end{align*}
and
\begin{align*}
Y' = \sin z \, \partial_x + \cos z \, \partial_y + \frac{1}{2} ( x \cos z - y \sin z ) \, \partial_z.
\end{align*}
Therefore
\[
X^2 + Y^2 =  \frac{\partial^2}{\partial x^2} + \frac{\partial^2}{\partial y^2} + \frac{1}{4} (x^2 + y^2) \frac{\partial^2}{\partial z^2} - \frac{1}{2} y \frac{\partial^2}{\partial x \partial z} + \frac{1}{2} x \frac{\partial^2}{\partial y \partial z}
\]
and
\[
(X')^2 + (Y')^2 = X^2 + Y^2 + \tfrac{1}{2}\, x\, \partial_x + \tfrac{1}{2}\, y\, \partial_y.
\]
In particular, $X^2+Y^2 \neq (X')^2 +(Y')^2$.
\end{example}

Observe that this example illustrates that there is little chance of recovering a statement such as Theorem \ref{thm.leftinvsumsq} in a more general setting, where left-invariance has no analogue. However, when we are fortunate enough to have Lie structure, we get as a corollary the following.
\begin{thm}\label{t.5.9}
Assume that both Assumptions \ref{property.HLI} and \ref{property.metricLI} hold and let $\Delta$ be a sub-Laplacian on $G$. Then there is a unique smooth vector field $X_{\Delta}$ such that given any orthonormal horizontal frame $\{X_1, ..., X_m\}$ of left invariant vector fields,
\[ \Delta = \sum\limits_{k=1}^m X_k^2 + X_{\Delta}.\]
\end{thm}
\begin{proof}
We established in Proposition \ref{prop.sumsqbetaij} that $\Delta = \sum\limits_{k=1}^m X_k^2 + \text{ first order terms}$. Let $D_{\Delta} = \Delta - \sum\limits_{k=1}^m X_k^2$.  If $\{Y_1, ..., Y_m\}$ is another orthonormal horizontal frame of left invariant vector fields, then Theorem \ref{thm.leftinvsumsq} implies that $D_{\Delta} = \Delta - \sum\limits_{k=1}^m Y_k^2$. From this, the conclusion follows.
\end{proof}

As we see below, meaningful choices for an analogue of the Laplace-Beltrami operator on a sub-Riemannian manifold are sub-Laplacians. However, as our work thus far illustrates, there is no debate about what the second order terms should be, rather it is the first order terms that distinguish one choice from another.

\section{$\diverg^{\w}\gradH$ and the sum of squares operators}\label{sec: divgrad}

\begin{df}
The \emph{horizontal gradient} of a smooth function $f: M \to \mathbb{R}$, denoted $\gradH f$,  is a horizontal vector field defined such that for all $X \in \mathcal{H}$,
\[
\langle \gradH f, X\rangle = X(f).
\]
\end{df}
One can readily check that $\gradH f = \beta(df)$ where $df$ is the standard exterior derivative; that is, $\gradH f$ is the \emph{horizontal dual} of $df$. From this, it follows that locally $\gradH f = \beta^{ij} \frac{\partial f}{\partial x^i} \frac{\partial}{\partial x^j}$. Moreover, given an orthonormal horizontal frame $\{X_1, ..., X_m\}$,
\[
\gradH f = \sum\limits_{j=1}^m X_j(f) X_j.
 \]

Assume that $M$ is orientable and $\w$ is some volume form on $M$ locally given by $\w = \tau dx^1 \wedge \cdots \wedge dx^d$; here $\tau :  M \to \mathbb{R}$ is positive and smooth. Using standard results in geometry, the divergence of a vector field $X$ with respect to $\w$ is
\begin{equation}
\diverg^{\w}(X) = \sum\limits_{i=1}^d \left\{ \frac{X^i}{\tau}\, \frac{\partial \tau}{\partial x^i} + \frac{\partial X^i}{\partial x^i}\right\} \label{eqn.divagainstform}
\end{equation}
Replacing $X$ with $\gradH f$, we find
\begin{align*}
\diverg^{\w}(\operatorname{grad}_{\mathcal{H}}f) = \sum\limits_{i, j =1}^d \left\{ \frac{\beta^{ij}}{\tau}\, \frac{\partial \tau}{\partial x^i} \frac{\partial}{\partial x^j} + \beta^{ij}\frac{\partial^2}{\partial x^i \partial x^j} + \frac{\partial \beta^{ij}}{\partial x^i} \frac{\partial}{\partial x^j}\right\} f
\end{align*}
which yeilds the following local formula for the operator $\diverg^{\w} \operatorname{grad}_{\mathcal{H}}$,
\begin{equation} \label{eqn.divgradH1}
\diverg^{\w}\operatorname{grad}_{\mathcal{H}} = \sum\limits_{i,j=1}^d \left\{ \beta^{ij} \frac{\partial^2}{\partial x^i \partial x^j} + \left[ \frac{\beta^{ij}}{\tau}\, \frac{\partial \tau}{\partial x^i} + \frac{\partial \beta^{ij}}{\partial x^i} \right] \frac{\partial}{\partial x^j} \right\}
\end{equation}
Comparing \eqref{eqn.divgradH1} with Corollary \ref{cor.LLclass} immediately leads to

\begin{cor}
$\diverg^{\w} \gradH$ is a sub-Laplacian.
\end{cor}
In particular, we can consider \eqref{eqn.divgradH1} in the case when $\w$ is a Riemannian volume form. Suppose that $( \cdot, \cdot)$ is some Riemannian metric on $M$ and $g : TM \to T^*M$ is the induced bundle isomorphism. As usual, we write $g_{ij} = \big( \frac{\partial}{\partial x^i}, \frac{\partial}{\partial x^j}\big)$, and further let the raised indices $g^{ij}$ be the entries of the matrix inverse of $(g_{ij})$. The Riemannian volume induced by this metric is the form
locally given by $\w = \sqrt{|g|} \, dx^1 \wedge \cdots \wedge dx^d$, where $|g| = \det(g_{ij})$.
In this setting, we can rewrite \eqref{eqn.divgradH1} as
\begin{equation} \label{eqn.divggradH}
\begin{aligned}
\operatorname{div}^g \operatorname{grad}_{\mathcal{H}} &= \sum_{l,k =1}^d\left\{ \beta^{lk} \,\frac{\partial^2}{\partial x^l \partial x^k} + \frac{\partial \beta^{lk}}{\partial x^l}\, \frac{\partial}{\partial x^k}-\frac{1}{2} \sum_{i,j =1}^d \beta^{lk} g_{ij} \frac{\partial g^{ij}}{\partial x^l}\, \frac{\partial}{\partial x^k}\right\} \\
& =  \sum_{l,k =1}^d\left\{ \beta^{lk} \,\frac{\partial^2}{\partial x^l \partial x^k} + \bigg[  \frac{\partial \beta^{lk}}{\partial x^l} -\frac{1}{2} \sum\limits_{i, j =1}^d \beta^{lk} g_{ij} \frac{\partial g^{ij}}{\partial x^l} \bigg] \frac{\partial}{\partial x^k} \right\}
\end{aligned}
\end{equation}
where we write $\diverg^g$ rather than $\diverg^{\w}$ to emphasize that we are using the Riemannian volume form with respect to the metric defined by $g$.

From here we have a good starting point to approach a reasonable definition of an analogue of the Laplace-Beltrami operator through a ``divergence of the gradient'' type construction; however, this will only be meaningful if there is some volume measure on $M$ to which we want to calculate a divergence with respect to. A priori, there are (at least) a couple intrinsic measures that we can put on these spaces; most commonly considered are the Hausdorff and Popp's measures. For a detailed description of Popp's measure see \cite{MontgomeryBook2002}. In the case that $M$ is a Lie group and there exists a global orthonormal horizontal frame of left invariant vector fields, then the Hausdorff and Popp's measure agree with the left Haar measure up to some scaling constant. We consider the Lie group setting presently.

Here we would like to make a comment about the choices implicitly made when we choose a reference measure. This is specific to the Lie group case, and it is not so easy to see in a general sub-Riemannian setting. While several authors (mentioned elsewhere in the current paper) considered these three measures, namely, the Hausdorff measure, the Haar measure and Popp's measure, they do not always indicate that the choice of left- \emph{or} right- invariant vector fields is significant not only for the Haar measure, but also for the Hausdorff measure and Popp's measure. Indeed, the significance of this choice is apparent when we look at the construction of Popp's measure. As to the Hausdorff measure, being a metric space measure it uses the Carnot-Caratheodory metric defined by \eqref{e.3.1}. It might not be obvious, but this metric is left- or -right invariant depending on our choices at the level of the Lie algebra.

\subsection{When $M$ is a Lie Group}

Again, let $M = G$ be a Lie Group on which we will assume both Properties \ref{property.HLI} and \ref{property.metricLI} hold. Let  $\mathcal{X} = \{X_1, ..., X_m\}$ be a left-invariant orthonormal horizontal frame. Denote by $\mu_L$ and $\mu_R$ the left and right Haar measures, respectively.

If we extend $\mathcal{X}$ to a full frame of $TG$ of left invariant vector fields $\{X_1, ..., X_m,$ $X_{m+1}, ..., X_d\}$ and let $\{\chi^1, ..., \chi^d\}$ be the corresponding dual frame, then the volume form $\chi^1 \wedge \cdots \wedge \chi^d$ is left-invariant and hence induces a left Haar measure. Since left (resp. right) Haar measure is unique up to a scalar multiple, constructing the left Haar measure in this way is independent of the extended frame up to this scalar multiple.  In particular, the divergence against $\chi^1 \wedge \cdots \wedge \chi^d$ is independent of our choice of an extension. From \cite{AgrachevBoscainGauthierRossi2009} (with the sign corrected) we have the following theorem.

\begin{thm} \label{thm.divLgradH}
Suppose that $\{X_1, ..., X_m\}$ is an orthonormal horizontal frame of left invariant vector fields. Let $\Delta^L = \diverg^{\mu_L} \gradH$. Then, using the notation introduced in Theorem \ref{t.5.9},
\begin{equation*}
X_{\Delta^L} = -  \sum\limits_{k=1}^m\operatorname{Tr}( \operatorname{ad} X_k(e))  X_k,
\end{equation*}
where $\operatorname{Tr}( \operatorname{ad} X_k(e))$ is the trace of the linear map defined by $\operatorname{ad}X_k(e)\,(v) = [X_k(e),v]$ for all $v \in \mathfrak{g}$.
This means that
\begin{equation}
\diverg^{\mu_L} \gradH = \sum\limits_{k=1}^m X_k^2 - \sum\limits_{k=1}^m\operatorname{Tr}( \operatorname{ad} X_k(e))  X_k.
\label{eqn.divLgradH}
\end{equation}
Moreover,  $G$ is unimodular if and only if $X_{\Delta^L} \equiv 0$, in which case
\[
\diverg^{\mu_L} \gradH  = \sum\limits_{k=1}^m  X_k^2.
\]
\end{thm}

The classification of unimodularity in terms of $X_{\Delta^L}$ can be found in \cite[Propositions 17, 18]{AgrachevBoscainGauthierRossi2009}. The derivation of an expression for $X_{\Delta^L}$ can be found in the same paper; we also provide a derivation below in Section \ref{s.calcs}. The calculation uses the standard fact that the divergence $\diverg^{\mu_L}(X)$ of a vector field $X$ can be found as
\[
\diverg^{\mu_L}(X)\, \chi^1 \wedge \cdots \wedge \chi^d =
\mathscr{L}_X ( \chi^1 \wedge \cdots \wedge \chi^d ) = d \circ \iota_X (\chi^1 \wedge \cdots \wedge \chi^d)
\]
where $\mathscr{L}_X$ is Lie differentiation along $X$, $d$ is exterior differentiation, and $\iota_X$ is interior multiplication with respect to $X$. Upon replacing $X$ with $\gradH f$ for some smooth map $f : M \to \mathbb{R}$, one arrives at
\[
\ d \circ \iota_{\gradH f} (\chi^1 \wedge \cdots \wedge \chi^d) = \Big\{ \sum\limits_{k=1}^m \Big( X_k^2 - \operatorname{Tr}( \operatorname{ad} X_k(e))  X_k \Big) f \Big\}\, \chi^1 \wedge \cdots \wedge \chi^d.
\]

From this, we can derive a similar expression for $\diverg^{\mu_R} \gradH$. We let $\mathfrak{m} : G \to (0,\infty)$ be the modular function and $\mathfrak{m_i} : G \to (0,\infty)$ be defined by $\mathfrak{m_i}(x) = \mathfrak{m}(x^{-1})$. It is well known that  $\mathfrak{m_i}$ is a continuous group homomorphism from $G$ into the multiplicative group $(0,\infty)$ (the same is true for $\mathfrak{m}$) and thus smooth, and moreover $\mu_R(dx) = \mathfrak{m_i}(x)\mu_L(dx)$.  The fact that $\mathfrak{m_i}$ is a homomorphism further implies that $\mathfrak{m_i}(x) \mathfrak{m}(x) = 1$ for every $x \in G$.

\begin{thm}\label{thm.divRgradH}
Suppose that $\{X_1, ..., X_m\}$ is an orthonormal horizontal frame of left invariant vector fields. Let $\Delta^R = \diverg^{\mu_R} \gradH$. Then, using the notation introduced in Theorem \ref{t.5.9},
\[
X_{\Delta^R} = \sum\limits_{k=1}^m \Big[ \mfm X_k(\mfmi) - \operatorname{Tr}( \operatorname{ad} X_k(e))\Big]  X_k.
\]
This means that
\[
\diverg^{\mu_R} \gradH = \sum\limits_{k=1}^m X_k^2 + \sum\limits_{k=1}^m \Big[ \mfm X_k(\mfmi) - \operatorname{Tr}( \operatorname{ad} X_k(e))\Big]  X_k.
\]
\end{thm}
\begin{proof}
As noted above, $\mfmi\, \chi^1 \wedge \cdots \wedge \chi^d$ induces the right Haar measure $\mu_R$. Therefore, we have
\begin{align*}
&d \circ \iota_X (\mfmi\, \chi^1 \wedge \cdots \wedge \chi^d) = d \big[\mfmi \, \iota_X( \chi^1 \wedge \cdots \wedge \chi^d) \big] \\
&\quad = d\mfmi \wedge \iota_X(\chi^1 \wedge \cdots \wedge \chi^d) + \mfmi \, d\circ \iota_X ( \chi^1 \wedge \cdots \wedge \chi^d).
\end{align*}
The second term in the last equality is readily  understood from the calculations with respect to $\mu_L$. Indeed, replacing $X$ with $\gradH f$ we have
\begin{align*}
 \mfmi \, d\circ \iota_{\gradH f} ( \chi^1 \wedge \cdots \wedge \chi^d) =  \Big\{ \sum\limits_{k=1}^m \Big( X_k^2 - \operatorname{Tr}( \operatorname{ad} X_k(e))  X_k \Big) f \Big\}\,\mfmi \chi^1 \wedge \cdots \wedge \chi^d
\end{align*}

 For the first term we get
\begin{align*}
\sum_{j,k = 1}^d \Big[ (-1)^{j+1} \big(X_k(\mfmi)  \chi^k \big) \wedge \big( \chi^j(X) \chi^1 \wedge \cdots \wedge \chi^{j-1} \wedge \chi^{j+1} \wedge \cdots \wedge \chi^d\big)\Big] \\
= \sum_{j,k = 1}^d \Big[ (-1)^{j+1} X_k(\mfmi)   \chi^j(X) \, \big(  \chi^k  \wedge \chi^1 \wedge \cdots \wedge \chi^{j-1} \wedge \chi^{j+1} \wedge \cdots \wedge \chi^d\big)\Big]\\
=\Big\{ \sum_{j,k = 1}^d  (-1)^{j+1} (-1)^{k+1} \delta_{jk}X_k(\mfmi)   \chi^j(X) \, \Big\}\,  \chi^1 \wedge \cdots \wedge \chi^d\\
= \Big\{ \mfm \sum_{k=1}^d X_k(\mfmi) \chi^k(X) \Big\} \,\mfmi\chi^1 \wedge \cdots \wedge \chi^d,
\end{align*}
where in the last equality we used the fact that the pointwise product $\mfm \mfmi = 1$. Replacing $X$ with $\gradH f = \sum\limits_{j=1}^m X_j(f) X_j$, we find
\[
d\mfmi \wedge \iota_{\gradH f}(\chi^1 \wedge \cdots \wedge \chi^d) = \Big\{ \mfm \sum_{k=1}^m X_k(\mfmi) X_k(f) \Big\} \,\mfmi\chi^1 \wedge \cdots \wedge \chi^d
\]
It is important to note that the last equality has the coefficient summing only through the $m$ horizontal vector fields. Combining these calculations leads to
\begin{align*}
&d \circ \iota_{\gradH f} (\mfmi\, \chi^1 \wedge \cdots \wedge \chi^d) \\
&\quad = \Big\{ \sum\limits_{k=1}^m \Big( X_k^2 + \big[ \mfm X_k(\mfmi) - \operatorname{Tr}( \operatorname{ad} X_k(e))\big]  X_k \Big) f \Big\} \, \mfmi\, \chi^1 \wedge \cdots \wedge \chi^d
\end{align*}
resulting in
$\displaystyle
\diverg^{\mu_R} \gradH =  \sum\limits_{k=1}^m X_k^2 + \sum\limits_{k=1}^m \Big[ \mfm X_k(\mfmi) - \operatorname{Tr}( \operatorname{ad} X_k(e))\Big]  X_k.
$
\end{proof}

\begin{rem}
Unlike $X_{\Delta^L}$ introduced in Theorem \ref{thm.divLgradH}, it can happen that $X_{\Delta^R} = 0$ when $G$ is not unimodular. In Example \ref{example.affine} we see that such is the case for the affine group. This asymmetry  stems from the fact that expressions for $X_{\Delta^L}$ and $X_{\Delta^R}$ are written in terms of left-invariant vector fields. 
\end{rem}

\begin{rem}
As previously mentioned, if we consider a left invariant structure on $G$ it can be natural to endow $G$ with a right Haar measure. In particular, the sum of squares $\sum\limits_{k=1}^m X_k^2$ of a left invariant orthonormal  horizontal frame is essentially self-adjoint with respect to the right Haar measure on $C_c^{\infty}(M)$; see  \cite[p. 950]{DriverGrossSaloff-Coste2009a}.
\end{rem}

\section{The Operator $\LV$}

We assume that the manifold $M$ is complete with respect to the metric $d_{CC}$.

\begin{nota}
We let $\Phi$ be the flow of the Hamilton-Jacobi equations \eqref{eqn.HamiltonsEq}. Indeed, we will consider $\Phi$ as a map
\[
\Phi : [0, \infty) \times T^*M \longrightarrow T^{\ast}M,\]
 such that if $X \in \mathcal{H}_x$ then $t \mapsto \Phi_t(x,p)$ is the curve $(x(t), p(t))$ in $T^{\ast}M$ satisfying Hamilton-Jacobi equations with initial conditions $x(0) = x$ and $p(0) = p$.
\end{nota}

\begin{rem}
The fact that for each choice of initial conditions $(x, p) \in T^*M$, the flow $t \mapsto \Phi_t(x,p)$ is defined for all $t \geqslant 0$ comes along with the assumption that $M$ is complete with respect to $d_{CC}$; see Theorem \ref{thm.HopfRinow}.
\end{rem}

Before defining the operator $\LV$ in Definition \ref{defi.Lg} below, we will use the following proposition, which is proved fiberwise in Proposition \ref{prop.distinguishedVlinear}.
\begin{prop}\label{prop.distinguishedV}
Suppose that $\V$ is a smooth sub-bundle of $TM$ such that $TM = \Horz \oplus \V$. Then there exists a unique symmetric, positive semi-definite linear map $g^{\V} : TM \to T^*M$ such that $\beta \circ g^{\V}(X) = X$ for every horizontal vector $X$, and $g^{\V}(Y) = 0$ for every $Y \in \V$.  If $( \cdot, \cdot )$ is a Riemannian metric extending the sub-Riemannian metric in such a way that $\V = \Horz^{\perp}$, then $g^{\V}(X) = g(X)$ for every horizontal $X$, where $g : TM \to T^*M$ is the bundle isomorphism induced by the Riemannian metric $(\cdot, \cdot)$. Further, $g^{\V} = g \circ \beta \circ g$.
\end{prop}

\begin{df}
We will call such a bundle $\V$ a (choice of) \emph{vertical distribution}. 
\end{df}

\begin{rem}
For a smooth manifold $M$ of dimension $2n +1$, a \emph{contact form} $\w$ is a one form on $M$ such that $\w \wedge (d \w)^n \neq 0$ where $(d \w)^n = d\w \wedge \cdots \wedge d\w$. If a contact form exists on $M$, then $M$ is necessarily orientable since $\w \wedge (d\w)^n$ is a nowhere vanishing $2n+1$ form. When $M$ is endowed with a contact form $\w$, then $(M, \w)$ is called a \emph{contact manifold}. There is a canonical horizontal distribution $\Horz$ of dimension $2n$ on a contact manifold $(M, \w)$ given by $\Horz = \ker(\w)$. Moreover, there is a canonical vertical vector field $T$, called the \emph{Reeb vector field}, defined by $\w(T) = 1$ and $\mathscr{L}_T \w = 0$, where $\mathscr{L}_T$ is the Lie derivative with respect to $T$. In particular, on such manifolds there is a meaningful and natural choice of vertical bundle $\V = \operatorname{span}(T)$.
\end{rem}

\begin{nota}
We denote the unit sphere in $\mathcal{H}_x$ by $\mathcal{S}^{\mathcal{H}}_x := \{ X \in \mathcal{H}_x : \langle X, X \rangle_x = 1\}$. The (unique) rotationally invariant measure on $\mathcal{S}_x$ will be denoted $\mathbb{U}_x$.
\end{nota}

\begin{df} \label{defi.Lg}
Define $\mathcal{L}^{\mathcal{V}} : C^{\infty}_c(M) \to \mathbb{R}$ as the second order operator defined by
\begin{equation} \label{eqn.LV}
 \LV f (x) := \int_{\mathcal{S}^{\mathcal{H}}_x} \left\{ \frac{d^2}{dt^2} \Big|_0 f \left(\Phi_{t}(x,g^{\V}(X))\right) \right\} \mathbb{U}_x(dX).
\end{equation}
\end{df}
The operator $\LV$ has been introduced in \cite{GordinaLaetsch2014a}, where it is shown that $\LV$ is the generator of a process which is the limit of a naturally constructed horizontal random walk. The operator $\LV$ can be viewed as the generator of a horizontal Brownian motion on $M$, the role played by the Laplace-Beltrami operator on Riemannian manifolds. The compelling notion here is that $\mathcal{L}^{\V}$ is introduced to be canonical with respect to a sub-Riemannian Brownian motion, whose construction depends only on a choice of vertical bundle $\V$, rather than on a choice of measure.

We give here a version of \cite[Theorem 3.5]{GordinaLaetsch2014a}, expressing $\LV$ in local coordinates. In comparison with \eqref{eqn.LB}, it becomes immediately clear that $\LV$ is the ($1/m$ scaled) Laplace-Beltrami operator in the Riemannian case $\mathcal{H} = TM$.
\begin{thm} \label{thm.LVlocal}
In local coordinates, $\LV$ can be written as
\begin{equation}
\label{eqn.LVlocal}
\LV = \frac{1}{m}  \sum\limits_{i,j=1}^d\left[ \beta^{ij} \,\frac{\partial^2}{\partial x^i \partial x^j} - \sum\limits_{k=1}^d  \Gamma^{ijk} [g^{\V}]_{ij}\, \frac{\partial}{\partial x^k} \right]
\end{equation}
where $g^{\V}$ was defined in Proposition \ref{prop.distinguishedV} and
\begin{equation}\label{eqn.raisedChristSR}
\Gamma^{ijk} = -\frac{1}{2} \sum\limits_{l=1}^d\left[ \beta^{il} \frac{\partial \beta^{jk}}{\partial x^l} + \beta^{jl} \frac{\partial \beta^{ik}}{\partial x^l} - \beta^{kl}\frac{\partial \beta^{ij}}{\partial x^l}\right]
\end{equation}
is the sub-Riemannian analogue of \eqref{eqn.raisedChrisR}.
\end{thm}
\begin{proof}
Let $(\cdot, \cdot)$ be any Riemannian metric on $M$ extending the sub-Riemannian metric in such a way that $\mathcal{V}$ is the orthogonal compliment of $\mathcal{H}$ with respect to this metric. Denote by $g : TM \to T^*M$ the bundle isomorphism induced by this extended metric, locally realized as a matrix with components $g_{ij} = \big( \frac{\partial}{\partial x^i}, \frac{\partial}{\partial x^j}\big)$. Theorem \cite[Theorem 3.5]{GordinaLaetsch2014a} gives the local formula for $\LV$
\begin{equation} \label{eqn.LVlocal2}
\LV = \frac{1}{m}  \sum\limits_{i,j=1}^d\left[ \beta^{ij} \,\frac{\partial^2}{\partial x^i \partial x^j} - \sum\limits_{a,b,k=1}^d  \Gamma^{ijk} g_{ia}\beta^{ab} g_{bj}\, \frac{\partial}{\partial x^k} \right].
\end{equation}
From Proposition \ref{prop.distinguishedV}, $g^{\V} = g \circ \beta \circ g$, from which we see $\sum\limits_{a, b=1}^d g_{ia}\beta^{ab} g_{bj} = [g^{\V}]_{ij}$, whence we conclude the result.
\end{proof}

\begin{rem}\label{r.6.9} The Riemannian metric $g$ extending the sub-Riemannian metric is sometimes called compatible (with the sub-Riemannian structure). This is the term we used in \cite{GordinaLaetsch2014a}.
\end{rem}

From Theorem \ref{thm.LVlocal} and its proof, we arrive at two corollaries. The first emphasizes how the selection of a compatible metric in \cite{GordinaLaetsch2014a} changes the first order term of $\LV$ and, moreover, how this compatible metric can be used as a tool in making calculations of $\LV$ tractable.

\begin{cor}\label{cor.calcLV}
Let $(\cdot, \cdot)$ be any Riemannian metric on $M$ extending the sub--Riemannian metric, and suppose that $\mathcal{V}$ is the orthogonal compliment of $\mathcal{H}$ with respect to this metric. In local coordinates, let $G$ be the matrix $G_{ij} = \big(\frac{\partial}{\partial x^i}, \frac{\partial}{\partial x^j} \big)$, and $B$ be the matrix with entries $\beta^{ij}$. Then $[g^{\V}]_{ij} = [ GBG]_{ij}$. In particular, according to \eqref{eqn.LVlocal},
\[
\LV = \frac{1}{m}  \sum\limits_{i,j=1}^d\left[ \beta^{ij} \,\frac{\partial^2}{\partial x^i \partial x^j} - \sum\limits_{k=1}^d  \Gamma^{ijk} [GBG]_{ij}\, \frac{\partial}{\partial x^k} \right]
\]
can be found in terms of the matrix $B$, its derivatives, and $G$. Moreover, only the first order term of $\LV$ depends on the extended metric, and any other extended metric such that $\V$ stays the orthogonal compliment of $\mathcal{H}$ gives rise to the same sub-Laplacian $\LV$.
\end{cor}
The second corollary of Theorem \ref{thm.LVlocal} follows immediately from Corollary \ref{cor.LLclass}.

\begin{cor}
Let $\Delta =  m \, \LV$. Then $\Delta$ is a sub-Laplacian.
\end{cor}

\subsection{Orthogonal Projection Along $\V$ and Comparison of $\LV$ with $\diverg^{\w} \gradH$}

Let $\mathcal{V}$ a choice of vertical bundle and $(\cdot,\cdot)$ be any Riemannian metric extending the sub-Riemannian metric which admits $\mathcal{V}$ as the orthogonal compliment of $\mathcal{H}$. As usual, denote by $g : TM \to T^*M$ the bundle isomorphism induced by the extended metric.

\begin{prop} \label{prop.gbetaorthproj}
The operator $\mathscr{P} := \beta \circ g : TM \to TM$ is orthogonal projection onto $\mathcal{H}$ along $\V$. Symmetrically, the operator $\mathscr{Q} := g \circ \beta : T^*M \to T^*M$ is orthogonal projection onto $g(\mathcal{H})$ along $\Null(\beta)$. Moreover, $\mathscr{P} \circ \beta =\beta \circ \mathscr{Q} =  \beta$ and $\mathscr{Q} \circ g = g \circ \mathscr{P} = g^{\V}$.
\end{prop}
\begin{proof}
Using the notation analogous to that introduced in the proof of Proposition \eqref{prop.distinguishedVlinear}, $g = \beta_{\mathcal{V}}^{-1} \oplus A : \mathcal{H} \oplus \V \to g(\mathcal{H}) \oplus \Null(\beta)$ and $\beta = \beta_{\mathcal{V}} \oplus {\bf 0} :  g(\mathcal{H}) \oplus \Null(\beta) \to \mathcal{H} \oplus \V$. Therefore, $\mathscr{P} = \operatorname{Id}_{\mathcal{H}} \oplus {\bf 0} : \mathcal{H} \oplus \V \to \mathcal{H} \oplus \V$ and $\mathscr{Q} = \operatorname{Id}_{g(\mathcal{H})} \oplus {\bf 0} : g(\mathcal{H}) \oplus \Null(\beta) \to g(\mathcal{H}) \oplus \Null(\beta)$.
\end{proof}

Continuing with the notation of Proposition \ref{prop.gbetaorthproj}, we express the first order term of $\LV$ in terms of $\mathscr{P}$. From \eqref{eqn.LVlocal2} and \eqref{eqn.raisedChristSR} and the symmetry of $\beta$ and $g$, the coefficient of $\partial/ \partial x^k$ in $\LV$ is
\begin{align*}
-\sum_{i,j=1}^d \Gamma^{ijk} [g^{\V}]_{ij} = \sum\limits_{i,j,l = 1}^d \Big( \sum\limits_{a,b =1}^d \beta^{i l } g_{i a} \beta^{a b} g_{b j} \frac{\partial \beta^{j k}}{\partial x^l} - \frac{1}{2} \beta^{l k}  [g^{\V}]_{ij} \frac{\partial \beta^{ij}}{\partial x^l}\Big) \\
= \sum\limits_{j,l,a=1}^d {\mathscr{P}^l}_a\,   {\mathscr{P}^a}_j  \frac{\partial \beta^{j k}}{\partial x^l} - \frac{1}{2} \sum\limits_{i,j,l=1}^d \beta^{l k}  [g^{\V}]_{ij}  \frac{\partial \beta^{ij}}{\partial x^l} \\
 = \sum\limits_{j,l=1}^d {\mathscr{P}^l}_j  \frac{\partial \beta^{j k}}{\partial x^l} - \frac{1}{2} \sum\limits_{i,j,l=1}^d \beta^{l k} \,  [g^{\V}]_{ij}  \frac{\partial \beta^{ij}}{\partial x^l}
\end{align*}
Here the final equality comes from the fact $\mathscr{P}^2 = \mathscr{P}$. Rearranging these terms and considering \eqref{eqn.divgradH1}, we get the following result.

\begin{thm} \label{thm.LV=divgradH}
Suppose that $M$ is oriented. There exists a volume form $\w = \tau\, dx^1 \wedge \cdots \wedge dx^d$ on $M$ such that $\LV = \frac{1}{m} \diverg^{\w} \gradH$ if and only if
\begin{equation}\label{eqn.LV=divgradH}
 \sum_{l=1}^d \left[-\frac{1}{2} \beta^{lk}  \sum\limits_{i,j=1}^d [g^{\V}]_{ij} \frac{\partial \beta^{ij}}{\partial x^l}  + \sum\limits_{j=1}^d {\mathscr{P}^l}_j \frac{\partial \beta^{jk}}{\partial x^l} \right] = \sum_{l=1}^d\left[ \beta^{lk}\Big(\frac{1}{\tau}\, \frac{\partial \tau}{\partial x^l}\Big) + \frac{\partial \beta^{lk}}{\partial x^l} \right].
\end{equation}
In particular, for the equality  $\LV = \frac{1}{m} \diverg^{\w} \gradH$, it is sufficient that both
\begin{equation} \label{eqn.LV=divgradH2}
-\frac{1}{2}\sum\limits_{i,j,l=1}^d [g^{\V}]_{ij} \frac{\partial \beta^{ij}}{\partial x^l}  = \frac{1}{\tau} \sum\limits_{l=1}^d\frac{\partial \tau}{\partial x^l} ~~\text{ and }~~ \sum\limits_{j,l=1}^d{\mathscr{P}^l}_j \frac{\partial \beta^{jk}}{\partial x^l} = \sum\limits_{l=1}^d \frac{\partial \beta^{lk}}{\partial x^l}.
\end{equation}
\end{thm}

For the affine group discussed in Example \ref{example.affine}, we see that the Riemannian volume of the standard extended metric agrees with the left Haar measure, however, $\LV$ agrees with the divergence of the gradient against the right Haar measure. In a certain sense, $\LV$ switches handedness in this case, illustrating that the interplay of $\LV$ and a choice of extended metric is not trivially reproducing the divergence of the gradient against the induced Riemannian volume.

\section{Examples} \label{s.examples}
We now demonstrate how to use the results of this paper for the Heisenberg group, $SU(2)$, and the affine group. Note that they represent three different models concerning topology (compact versus non-compact) and unimodularity.

\subsection{Heisenberg Group} \label{example.heisenberg}
Let $\mathbb{H}$ be the Heisenberg group; that is, $\mathbb{H}\cong\mathbb{R}^3$ with the multiplication defined by

\[
\left( x_{1}, y_{1}, z_{1} \right) \star \left( x_{2}, y_{2}, z_{2} \right):=\left( x_{1}+x_{2}, y_{1}+y_{2}, z_{1}+z_{2} + \frac{1}{2}  \omega\left( x_{1}, y_{1}; x_{2}, y_{2}\right)\right),
\]
where $\omega$ is the standard symplectic form

 \[
 \omega\left( x_{1}, y_{1}; x_{2}, y_{2}\right):=x_{1}y_{2} - y_{1}x_{2}.
 \]
We define $X$, $Y$, and $Z$ as the unique left-invariant vector fields with $X_e = \partial_x$, $Y_e = \partial_y$, and $Z_e = \partial_z$. We find
\begin{align*}
&X = \partial_x - \frac{1}{2} y \partial_z,\\
&Y = \partial_y + \frac{1}{2} x \partial_z,\\
&Z = \partial_z.
 \label{eqn.XYheis}
\end{align*}
The horizontal distribution is defined by $\mathcal{H} = \operatorname{span}\{X,Y\}$ (understood fiberwise). We check that $[X,Y] = Z$, so H\"{o}rmander's condition is easily satisfied.
We endow $\mathbb{H}$ with the sub-Riemannian metric $\langle \cdot, \cdot \rangle$ so that $\{X,Y\}$ is an orthonormal frame for the horizontal distribution. The group $\mathbb{H}$ is nilpotent, and therefore it is unimodular. Let $\mu$ be the Haar measure on $\mathbb{H}$.
\begin{lem}\label{lem.HeisHaar}
The Haar measure is given by $\mu = dx \wedge dy \wedge dz$.
\end{lem}
\begin{proof}
 By inspection, the dual basis $\{\chi^X, \chi^Y, \chi^Z\}$ of $\{X,Y,Z\}$ is
\begin{align*}
&\chi^X = dx,~~\chi^Y = dy,~~\chi^Z = \frac{1}{2} y \, dx - \frac{1}{2} x \, dy + dz
\end{align*}
and hence $\mu = \chi^X \wedge \chi^Y \wedge \chi^Z = dx \wedge dy \wedge dz$.
\end{proof}

\begin{prop}\label{prop.divgradHeis}
We have
$\displaystyle
\diverg^{\mu}\gradH = X^2 + Y^2.
$
\end{prop}
\begin{proof}
Since $\mathbb{H}$ is unimodular, this follows directly from Theorem \ref{thm.divLgradH}.
\end{proof}

\begin{prop}\label{p.7.3}
Let $\mathcal{V} = \operatorname{span}\{Z\}$ be the vertical distribution. Then

\[
\LV  = \frac{1}{2} \diverg^\mu\gradH.
\]
\end{prop}
\begin{proof} We present two proofs below. One is based on a direct computation of both operators in question, while the second proof is an application of Theorem \ref{thm.LV=divgradH}.

\emph{Proof 1.} Using \eqref{eqn.LVlocal}, it is shown in \cite{GordinaLaetsch2014a} that $\LV = \frac{1}{2}(X^2 +Y^2)$. Comparing this to Proposition \ref{prop.divgradHeis} yields the desired result, and hence $\LV = \frac{1}{2} \diverg^{\mu} \gradH$.

\emph{Proof 2.} Let $g$ be a metric which extends the sub-Riemannian metric so that $\{X,Y,Z\}$ is an orthonormal frame of $TM$; note that with respect to $g$, $\mathcal{V} = \mathcal{H}^{\perp}$. Let $B$ be the matrix $(\beta^{ij})$ and $G$ be the matrix $(g_{ij})$ in standard coordinates. We have
\[
B =
\begin{pmatrix}
1 & 0 & -\frac{y}{2} \\
0 & 1 & \frac{x}{2} \\
-\frac{y}{2} & \frac{x}{2} & \frac{x^2 + y^2}{4}
\end{pmatrix}
~~
\text{ and }
~~
G =
\begin{pmatrix}
1+ \frac{ y^2}{4} & -\frac{ xy}{4} & \frac{ y}{2} \\
-\frac{ xy}{4} & 1 + \frac{ x^2}{4} & -\frac{ x}{2} \\
\frac{ y}{2} & -\frac{ x}{2} & 1
\end{pmatrix}.
\]
Using Proposition \ref{prop.distinguishedV} and Corollary \ref{cor.calcLV}, we see that
\[
[g^{\V}] = GBG = \begin{pmatrix}
1 & 0 & 0 \\
0 & 1 & 0 \\
0 & 0 & 0
\end{pmatrix}.
\]
The matrix representing the projection $\mathscr{P} = \beta \circ g$ onto $\mathcal{H}$ along $\mathcal{V}$ is
\[
P = BG =
\begin{pmatrix}
1 & 0 & 0 \\
0 & 1 & 0 \\
-\frac{y}{2} & \frac{x}{2} & 0
\end{pmatrix}.
\]
From this, we have
\[
\sum\limits_{j,l=1}^3{P^l}_j \frac{\partial B^{jk}}{\partial x^l} = 0 =  \sum\limits_{l=1}^3 \frac{\partial B^{lk}}{\partial x^l}\]
and
\[
\sum\limits_{i,j,l=1}^3 [g^{\V}]_{ij} \frac{\partial \beta^{ij}}{\partial x^l} = 0 = \frac{1}{\tau} \sum\limits_{l =1}^3 \frac{\partial \tau}{\partial x^l}
\]
for any $\tau \equiv \text{constant}$. Therefore, \eqref{eqn.LV=divgradH2} is easily satisfied and we learn that $\LV = \frac{1}{2} \diverg^{\mu} \gradH$.
\end{proof}

\subsection{SU(2)} \label{example.SU2}
$SU(2)$ is a compact connected unimodular Lie group, diffeomorphic to the $3$-sphere $S^3$. One identification of $SU(2)$ is as the group under matrix multiplication of the following space of matrices
\[
SU(2) = \left\{ \begin{pmatrix} a & -\bar{b} \\ b & \bar{a} \end{pmatrix} : a, b \in \mathbb{C} \right\}
\]
We use Euler angles as our standard coordinates $\{\theta, \phi, \psi\}$ with the convention that $x^1 = \theta$, $x^2 = \phi$, and $x^3 = \psi$. While a typical convention is that the first and second coordinates are swapped from ours here, but we choose this convention to simplify the appearance of some of the later calculations.
Let $X$, $Y$, and $Z$ be given by
\begin{equation}\label{eqn.su2D1D2D3}
\begin{aligned}
&X = \cos \psi \, \partial_{\theta} + \frac{\sin \psi}{\sin \theta} \,\partial_{\phi} - \cos \theta \,\frac{\sin \psi}{\sin \theta} \, \partial_{\psi} \\
&Y = -\sin \psi \, \partial_{\theta} + \frac{\cos \psi}{\sin \theta} \, \partial_{\phi} - \cos \theta \, \frac{ \cos \psi}{\sin \theta} \, \partial_{\psi} \\
&Z = \partial_{\psi}.
\end{aligned}
\end{equation}
We define the horizontal distribution as $\mathcal{H} = \operatorname{span}\{X, Y\}$, and the sub-Riemannian metric $\langle \cdot, \cdot \rangle$ such that the collection $\{X,Y\}$ forms an orthonormal frame. Since $SU(2)$ is compact, it is unimodular. Let $\mu$ be the Haar measure on $SU(2)$.

\begin{lem} \label{lem.SU2Haar}
We have $\mu = \sin(\theta) \, d\theta \wedge d\phi \wedge d\psi.$
\end{lem}
\begin{proof}
 By inspection we find that the dual frame $\{\chi^X, \chi^Y, \chi^Z\}$ to $\{X,Y,Z\}$ is
\begin{align*}
&\chi^X = \cos\psi \, d\theta + \sin\theta \sin\psi \, d \phi,  \\
&\chi^Y = -\sin\psi \, d\theta + \sin\theta \cos\psi \, d\phi,  \\
&\chi^Z = \cos \theta \, d\phi + d\psi,
\end{align*}
and hence
$\displaystyle \mu =  \chi^X \wedge \chi^Y \wedge \chi^Z = \sin\theta \, d\theta \wedge d\phi \wedge d\psi.$
\end{proof}

\begin{prop}
We have $\diverg^{\mu} \gradH = X^2 + Y^2$.
\end{prop}
\begin{proof}
Since $SU(2)$ is unimodular, this follows directly from Theorem \ref{thm.divLgradH}.
\end{proof}

\begin{prop}
Let $\mathcal{V} = \operatorname{span}\{Z\}$ be the vertical distribution. Then

\[
\LV = \frac{1}{2} \diverg^{\mu} \gradH.
\]
\end{prop}
\begin{proof}
As in the Heisenberg case, we present two proofs below. The first proof is based on a direct computation of both operators, while the second proof is an application of Theorem \ref{thm.LV=divgradH}.
 Let $\lambda > 0$ and $g = g_{\lambda}$ be a compatible Riemannian metric making $\{X, Y, Z \}$ orthogonal such that $\lambda : = g( Z, Z )$. The introduction of arbitrary $\lambda$, rather than fixing some value, say $\lambda = 1$, is purely pedagogical to illustrate how Proposition \ref{prop.distinguishedV} manifests within our calculation. Letting $B$ and $G$ be the $3 \times 3$ matrices representing $\beta$ and $g$ in the $\{\theta, \phi, \psi\}$ coordinates respectively,
\begin{equation} \label{eqn.su2B}
B = \begin{pmatrix}
1 & 0 & 0 \\
0 & \frac{1}{\sin^2 \theta} & - \frac{\cos \theta}{\sin^2 \theta} \\
0 & - \frac{\cos \theta}{\sin^2 \theta} & \frac{\cos^2 \theta}{\sin^2 \theta}
\end{pmatrix}
~~\text{ and }~~ G = \begin{pmatrix}
1 & 0 & 0\\
0 & \sin^2 \theta + \lambda \cos^2 \theta & \lambda \cos \theta \\
0 & \lambda \cos \theta & \lambda
\end{pmatrix}
\end{equation}

\emph{Proof 1.} First, using Proposition \ref{prop.distinguishedV} and Corollary \ref{cor.calcLV} we see that
\begin{equation}\label{eqn.su2GBG}
[g^{\V}] = GBG =  \begin{pmatrix}
1 & 0 & 0 \\
0 & \sin^2 \theta & 0 \\
0 & 0 & 0
\end{pmatrix}.
\end{equation}

From \eqref{eqn.raisedChristSR} it becomes apparent that the only non-zero term of $\Gamma^{ijk}$ when $i=j=1$ or $i=j=2$ is $\Gamma^{221} =  -\frac{\cos\theta}{\sin^3\theta}$. Hence
\[
\sum\limits_{i,j=1}^d \Gamma^{ijk} [g^{\V}]_{ij} \,  \frac{\partial}{\partial x^k} = \Gamma^{221}[g^{\V}]_{22}\, \partial_{\theta} = - \frac{\cos \theta}{\sin \theta} \, \partial_{\theta}.
\]
We therefore deduce
\begin{equation} \label{eqn.LVsu2}
\begin{aligned}
\mathcal{L}^{\V} &= \frac{1}{2} \sum\limits_{i,j=1}^d \left\{ \beta^{ij} \frac{\partial^2}{\partial x^i \partial x^j} - \sum\limits_{k=1}^d \Gamma^{ijk}[g^{\V}]_{ij} \frac{\partial}{\partial x^k} \right\} \\
& =\frac{1}{2} \left\{ \partial^2_{\theta} + \frac{1}{\sin^2 \theta} \partial^2_{\phi} + \frac{\cos^2 \theta}{\sin^2 \theta} \partial^2_{\psi} - 2 \frac{\cos \theta}{\sin^2 \theta} \partial_{\phi}\partial_{\psi} + \frac{\cos \theta}{\sin \theta} \partial_{\theta} \right\} \\
& = \frac{1}{2} (X^2 + Y^2)
\end{aligned}
\end{equation}

From Theorem \ref{thm.divLgradH} or \ref{thm.divRgradH}, this implies that if $\mu$ is the Haar measure, then $\diverg^{\mu} \gradH = X^2 + Y^2$. From \eqref{eqn.LVsu2}, it is clear that $\LV = \frac{1}{2} \diverg^{\mu} \gradH$.

\emph{Proof 2.} In local coordinates, the matrix of the projection $\mathscr{P} = \beta \circ g$ onto $\mathcal{H}$ along $\mathcal{V}$ is
\[
P = BG = \begin{pmatrix}
1 & 0 & 0 \\
0 & 1 & 0 \\
0 & - \cos \theta & 0
\end{pmatrix}
\]
Since $\mu = \sin(\theta) \,d\theta \wedge d\phi \wedge d\psi$, we easily check that \eqref{eqn.LV=divgradH2} is satisfied with $\tau = \sin \theta$. We have,
\begin{align*}
-\frac{1}{2} \sum\limits_{i,j, l = 1}^3 [GBG]_{ij} \frac{\partial B^{ij}}{\partial x^l} = -\frac{1}{2} \sin^2 \theta [ - 2 \sin^{-3} \theta \cos \theta]   = \frac{\cos \theta}{\sin\theta}  \delta_{\theta, x^l} \\
= \frac{1}{\sin \theta} \frac{\partial \sin \theta}{\partial \theta}  = \frac{1}{\tau} \sum\limits_{l=1}^3 \frac{\partial \tau}{\partial x^l},
\end{align*}
and clearly for any $l = 1, 2, 3$,
$\sum\limits_{j=1}^3 {P^l}_j\, \frac{\partial B^{ik}}{\partial x^l} = 0 = \frac{\partial B^{lk}}{\partial x^l}.
$ Therefore, from Theorem \ref{thm.LV=divgradH}, we deduce $\LV = \frac{1}{2} \diverg^{\mu} \gradH$.
\end{proof}

\subsection{Affine Group} \label{example.affine}
Let $G = (0,\infty) \times \mathbb{R}^2$ be the group with operation
\[ (a,b,c)\star(x,y,z) = (ax, ay+b, z + c).\]
It is easy enough to check that the identity is $e = (1,0,0)$. Moreover, $G$ is a Lie group with Lie algebra generated by
\[
\mathfrak{g} = \operatorname{span}\{ X|_e, Y|_e, Z|_e\}
\]
with $X|_e = \partial_x |_e$, $Y|_e = (\partial_y + \partial_z) |_e$, and $Z|_e = \partial_y |_e$. Extending these to left invariant vector fields, we have
\[
X(x,y,z) = x \partial_x, ~Y(x,y,z) = x \partial_y + \partial_z, ~\text{and}~Z(x,y,z) = x \partial_y.
\]
We give $G$ a sub-Riemannian structure by defining $\mathcal{H} = \operatorname{span}\{X, Y\}$ with inner-product $\langle \cdot, \cdot \rangle$ making $\{X,Y\}$ a (global) orthonormal frame.
Note that $[X,Y] = Z$, so H\"{o}rmander's condition is easily satisfied.
The affine group $G$ is not unimodular. Let $\mu_L$ and $\mu_R$ be the left and right Haar measures, respectively.

\begin{lem}
We have $\mu_L = x^{-2} dx\wedge dy \wedge dz$ and $\mu_R = x^{-1} dx\wedge dy \wedge dz$.
\end{lem}
\begin{proof}
By inspection, the dual frame $\{\chi^X,\chi^Y,\chi^Z\}$ to $\{X,Y,Z\}$ is
\begin{align*}
&\chi^X = x^{-1}\,dx,~~ \chi^Y = dz, ~~\chi^Z = x^{-1}\,dy - dz
\end{align*}
From this we find the left Haar measure $\mu_L = \chi^X \wedge \chi^Z \wedge \chi^Y = x^{-2} dx\wedge dy \wedge dz$. The analogous  calculation gives that the right Haar measure is $\mu_R = x^{-1}\,dx\wedge dy\wedge dz$.
\end{proof}

\begin{prop}
We have $\diverg^{\mu_L}\gradH = X^2 + Y^2 - X$ and $\diverg^{\mu_R} \gradH = X^2 + Y^2$.
\end{prop}
\begin{proof}
Using Theorem \ref{thm.divLgradH}, we deduce that
\[
X_{\Delta^L} = - \Big[  \Big(\chi^X[X,Y] + \chi^Z[X,Z]\Big) X + \Big( \chi^X[Y,X] + \chi^Z[Y,Z] \Big) \Big] = - X,
\]
showing that $\diverg^{\mu_L} \gradH = X^2+Y^2-X$.

For the right Haar measure, note that $d\mu_R = \mathfrak{m}_i d\mu_L$, implying that $\mathfrak{m}_i (x,y,z) = x$ and $\mathfrak{m}(x,y,z) = x^{-1}$. From Theorem \ref{thm.divRgradH},
\begin{align*}
X_{\Delta^R} = \mathfrak{m}\big[X(\mathfrak{m}_i) \,X + Y(\mathfrak{m}_i)\,Y \big] - X = x^{-1}[xX + 0] - X = 0,
\end{align*}
confirming that $\diverg^{\mu_R}\gradH = X^2+Y^2$.
\end{proof}

\begin{prop} \label{prop.LVaffine}
Let $\mathcal{V} = \operatorname{span}\{Z\}$ be the vertical distribution. Then

\[
\LV = \frac{1}{2} \diverg^{\mu_R} \gradH = \frac{1}{2}(X^2+Y^2).
\]
\end{prop}

\begin{proof} We omit the derivation of $\LV$ using \eqref{eqn.LVlocal} as we had in the previous two examples, and present the simplest confirmation of the result using Theorem \ref{thm.LV=divgradH}.
As before, extend the sub-Riemannian metric  to the Riemannian metric $g$ such that $\{X, Y, Z\}$ is an orthonormal frame. Letting $B$ and $G$ be the matrices representing $\beta$ and $g$ in standard coordinates respectively,
\begin{equation} \label{eqn.affineB}
B = \begin{pmatrix}
x^2 & 0 & 0 \\
0 & x^2 & x \\
0 & x & 1
\end{pmatrix}
~~\text{ and }~~
G = \begin{pmatrix}
x^{-2} & 0 & 0 \\
0 & x^{-2} & -x^{-1} \\
0 & -x^{-1} & 2
\end{pmatrix}
\end{equation}
Using Proposition \ref{prop.distinguishedV} and Corollary \ref{cor.calcLV},
\[
[g^{\V}] = GBG = \begin{pmatrix}
x^{-2} & 0 & 0 \\
0 & 0 & 0 \\
0 & 0 & -1
\end{pmatrix}.
\]
In local coordinates, the matrix representing the projection $\mathscr{P}= \beta \circ g$ onto $\mathcal{H}$ along $\mathcal{V}$ is
\[
P = BG = \begin{pmatrix}
1 & 0 & 0 \\
0 & 0 & x \\
0 & 0  & 1
\end{pmatrix}
\]
from which we find
\[
\sum\limits_{i, l =1}^3 {P^{l}}_j \frac{\partial B^{jk}}{\partial x^l} \delta_{1k} =  \frac{\partial B^{11}}{\partial x} \delta_{1k} = \sum_{l =1}^3 \frac{\partial B^{lk}}{\partial x^l}
\]
and
\[
-\frac{1}{2} \sum\limits_{i,j, l=1}^3 [g^{\V}]_{ij} \frac{\partial B^{ij}}{\partial x^l} =- \frac{1}{2} x^{-2} (2x)  = -\frac{1}{x} =  \frac{1}{\tau} \frac{\partial \tau}{\partial x} =   \sum\limits_{l=1}^3 \frac{1}{\tau} \frac{\partial \tau}{\partial x^l}
\]
when $\tau = x^{-1}$. Since we have shown \eqref{eqn.LV=divgradH2} is satisfied, we conclude that $\LV = \frac{1}{2} \diverg^{\mu_R} \gradH$ as $\mu_R = \tau \,dx\wedge dy \wedge dz$.
\end{proof}
\begin{rem}
The metric $g$ we used in the proof of Proposition \ref{prop.LVaffine} gives rise to the Riemannian volume
\[
\sqrt{|g|} \, dx\wedge dy \wedge dz = x^{-2} \, dx \wedge dy \wedge dz = \mu_L.
\]
This is interesting since $\LV$ gives the divergence of the horizontal gradient against the right Haar measure even though the Riemannian volume of the extended metric $g$ gives rise to the left Haar measure.
\end{rem}

\section{Appendix A: Derivation of \eqref{eqn.divLgradH}} \label{s.calcs}

\newcommand{\hux}[1]{\hat{\underline \chi}^{#1}}
\newcommand{\huX}[1]{\hat{\underline X}_{#1}}
\newcommand{\huuX}[1]{\hat{\underline{\underline X}}_{#1}}

We let $G$ be a Lie group which is also a sub-Riemannian manifold with horizontal distribution $\mathcal{H}$ admitting a global orthonormal frame of left invariant vector fields $\{X_1, ..., X_m\}$. Extend this to global  frame of left-invariant vector fields on $TM$,  $\{X_1, ..., X_m, X_{m+1}, ..., X_d\}$. Let $\{\chi^1, ..., \chi^d\}$ be the dual frame, in which case
\[
\mu = \chi^1 \wedge \cdots \wedge \chi^d
\]
is a left-invariant volume on $G$, and hence a scalar multiple of left Haar measure. It therefore suffices to show that $\diverg^{\mu} \gradH$ agrees with \eqref{eqn.divLgradH}.

 For some $1 \leq i \leq d$, by $\hux{i}$ we mean the $d-1$ form
\[
\hux{i} := \chi^1 \wedge \cdots \wedge \chi^{i-1} \wedge \chi^{i+1} \wedge \cdots \wedge \chi^d.
\]
Similarly, by $\huX{i}$ we mean the $d-1$ tuple
\[
\huX{i} := (X_1, ..., X_{i-1}, X_{i+1}, ..., X_d).
\]
Let's note that $\hux{i}(\huX{j}) = {\delta^{i}}_j$. Finally, for $1 \leq i < j \leq d$, we let $\huuX{i,j}$ be the $d-2$ tuple
\[
\huuX{i,j} := (X_1, ..., X_{i-1}, X_{i+1}, ..., X_{j-1}, X_{j+1}, ..., X_d).
\]
\begin{claim}
$\displaystyle d\hux{i} = (-1)^i \Big[ \sum\limits_{k=1}^d \chi^k([X_i,X_k]) \Big] \mu.$
\end{claim}
\begin{proof}
Note that $d\hux{i} = f\mu$ for some smooth function $f$ by a dimensionality argument. Therefore
\begin{align*}
f = d\hux{i}(X_1, ..., X_d) = \sum_{k=1}^d (-1)^{k-1} X_k \big( \hux{i}(\huX{k}) \big) + \sum\limits_{1 \leq k < j \leq d} (-1)^{k+j} \hux{i}([X_k,X_j], \huuX{k,j}).
\end{align*}

The first sum on the right hand side is $0$ as $\hux{i}(\huX{k}) = {\delta^i}_k$. Turning our focus onto the term, we realize that the only possible non-zero outcome will occur when either $k$ or $j$ is equal to $i$, since otherwise, $X_i$ will be one of the vector fields within the argument of $\hux{i}$, forcing a null result.

If $k=i$: Write $[X_i, X_j] = \sum\limits_m \chi^m\big([X_i,X_j]\big) X_m$. We have
\[
\hux{i}\big([X_i,X_j], \huuX{i,j}\big) = \sum_m \chi^m\big([X_i,X_j]\big)\, \hux{i}\big( X_m, \huuX{i,j}\big) = = \chi^j\big([X_i,X_j]\big) \, \hux{i} \big(X_j, \huuX{i,j} \big)
\]
where we used that if $m \neq j$, then we will have a repeated vector field in the argument of $\hux{i}$, again resulting in $0$. From here, we have
\[
\begin{aligned}
\big(X_j, \huuX{i,j} \big) = (-1)^{j-2} (X_1, ..., X_{i-1}, X_{i+1}, ..., X_{j-1}, X_j, X_{j+1}, ..., X_d) = \huX{i}
\end{aligned}
\]
Hence
$
\hux{i}( X_j, \huuX{i,j}) = (-1)^{j-2} \, \hux{i}(\huX{i}) = (-1)^j.
$
This finally results in
\[
\hux{i}\big([X_i,X_j], \huuX{i,j}\big) = (-1)^j \, \chi^j\big([X_i,X_j]\big).
\]

If $j=i$. Write $[X_k, X_i] = \sum\limits_m \chi^m\big([X_k,X_i]\big) X_m$ and working by the same argument in the previous case, we find
\[ \begin{aligned}
\hux{i}\big([X_k,X_i], \huuX{k,i}\big) = \sum_m \chi^m\big([X_k,X_i]\big)\, \hux{i}\big( X_m, \huuX{k,i}\big)
\\ = \chi^k\big([X_k,X_i]\big) \, \hux{i} \big(X_k, \huuX{k,i} \big)
\end{aligned} \]
and $(X_k, \huuX{k,i} \big) = (-1)^{k-1} \huX{i}$, implying $\hux{i}( X_k, \huuX{k,i})= (-1)^{k-1} \, \hux{i}(\huX{i}) = (-1)^{k-1}.$
This finally results in
\[
\hux{i}\big([X_k,X_i], \huuX{k,i}\big) = (-1)^{k-1}\, \chi^k\big([X_k, X_i]\big) = (-1)^k \,\chi^k\big([X_i,X_k]\big)
\]
Therefore
\begin{align*}
f &= \sum\limits_{1 \leq k < j \leq d} (-1)^{k+j} \hux{i}([X_k,X_j], \huuX{k,j}) \\
&= \sum\limits_{i < j \leq d} (-1)^{i+j} (-1)^j \, \chi^j\big([X_i,X_j]\big) + \sum\limits_{1 \leq k < i} (-1)^{k+j} (-1)^{k} \, \chi^k\big([X_i,X_k]\big) \\
&= (-1)^i \sum\limits_{k=1}^d  \chi^k\big([X_i,X_k]\big)
\end{align*}
which finishes the proof of the claim.
\end{proof}

\begin{proof}[Proof of \eqref{eqn.divLgradH}]
For any horizontal vector field $X = \sum\limits_{i=1}^m a^i X_i$, we have
\begin{align*}
\diverg^{\mu}(X)\mu = d \circ \iota_X(\mu) = \sum\limits_{i=1}^d (-1)^{i+1} \Big[ d \chi^i(X) \wedge \hux{i} + \chi^i(X) d\hux{i} \Big] \\
=\sum\limits_{i=1}^m (-1)^{i+1} da^i \wedge \hux{i} - \sum_{i=1}^m a^i \sum\limits_{k=1}^d \chi^k\big([X_i,X_k]\big) \mu
\end{align*}
where the second equality was established in the preceding claim. Note that $da^i = \sum\limits_{k=1}^d X_k(a^i) \chi^k$ and $\chi^k \wedge \hux{i} = (-1)^{i+1}  \delta_{ik}\, \mu$. We then deduce
\[
\diverg^{\mu}(X) =  \sum\limits_{i=1}^m \Big[ X_i(a^i) - \sum_{k=1}^d \chi^k\big([X_i,X_k]\big) a^i \Big]
\]
Replacing $a^i$ with $X_i(f)$, then $X = \gradH f$ and
\[
\diverg^{\mu}\gradH f  = \sum\limits_{i=1}^m \Big[ X_i^2  - \sum_{k=1}^d \chi^k\big([X_i,X_k]\big) X_i\Big] \,f
\]
We are done once we notice that $\sum\limits_{k=1}^d \chi^k\big([X_i,X_k]\big) = \operatorname{Tr}( \operatorname{ad} X_i(e))$ is defined on the Lie group independent of choice of extension of our orthonormal horizontal frame.
\end{proof}

\section{Appendix B: Linear Algebraic Preliminaries} \label{s.linear}
For this section, let $T$ be a finite dimensional vector space of dimension $d$. As is common, we let $T^{\ast}$ denote the dual space of $T$. Further, let $H \subset T$ be a subspace of dimension $m$.
\subsection{Inner products and the isomorphisms between $T$ and $T^*$ }\label{ss.TT*iso} An inner product $\langle \cdot, \cdot \rangle$ on $T$ induces a symmetric, positive definite isomorphism $g : T  \to T^*$ defined by $g(X) = \langle \cdot, X \rangle$.  The inverse map $\beta = g^{-1} : T^* \to T$ is the  symmetric, positive definite isomorphism defined by $\langle \beta(p), X \rangle = p(X)$ for every $X \in T$ and $p \in T^*$. In fact, you will recognize that $\beta$ is the isomorphism defined via the Hilbert space version of Riesz representation where $\beta(p) = Y \in T$ if and only if $p(X) = \langle X, Y \rangle$ for every $X \in T$.

Had we started with a symmetric, positive definite isomorphism $\beta : T^* \to T$, we can then recover an inner-product $\langle \cdot, \cdot \rangle$ on $T$ by $\langle X, Y \rangle = \eta \big( \beta(p) \big)$ where $\beta(\eta) = X$ and $\beta(p) = Y$. Note that the symmetry of $\beta$ is the statement that $\eta\big( \beta(p)\big) = p \big( \beta(\eta) \big)$ for every $p, \eta \in T^*$, and that positive definiteness of $\beta$ means $p\big(\beta (p)\big) > 0$ whenever $0 \neq p \in T^*$; from this and the linearity of $\beta$, it follows nearly immediately that $\langle \cdot, \cdot \rangle$ is an inner product. To summarize these well-known relations,

\begin{prop}
\label{prop.posdefisos}
There are canonical bijections between the following spaces.
\begin{enumerate}
\item[$\operatorname{IP}$:]  Inner products on $T$.
\item[$\operatorname{B}$:] Symmetric, positive definite isomorphisms $\beta : T^* \to T$.
\item[$\operatorname{G}$:] Symmetric, positive definite isomorphisms $g : T \to T^*$.
\end{enumerate}
The bijection between $\operatorname{IP}$ and $\operatorname{B}$ is defined by the equality $\langle \beta(p), X \rangle = p(X)$ for every $p \in T^*$ and $X \in T$; the bijection between $\operatorname{IP}$ and $\operatorname{G}$ is defined by the equality $\langle \cdot, X \rangle = g(X)$ for every $X \in T$; and the bijection between  $\operatorname{B}$ and $\operatorname{G}$ is defined by the equality $\beta = g^{-1}$.
\end{prop}
Note that if $\{X_1, ..., X_d\}$ is a basis of $T$ which is orthonormal with respect to the inner product $\langle \cdot, \cdot \rangle$, and if $\{ p^1, ..., p^d\}$ is its dual basis, then the corresponding map $\beta$ can be defined by $\beta(p^i) = X_i$ for each $1 \leq i \leq d$. This is readily confirmed by realizing that $\langle \beta(p^i), X_j \rangle = p^i(X_j) = \delta_{ij} = \langle X_i, X_j \rangle$, implying that $\langle \beta(p^i) - X_i, \cdot \rangle \equiv 0$.

\subsection{Indefinite Inner Products and $\beta$} We move now to the setting where, instead of an inner product being defined on all of $T$, an inner product is defined only on a subspace $H \subset T$. We will work to recover what we can from Proposition \ref{prop.posdefisos} in this setting; however, there is no canonical choice of symmetric linear map $g : T \to T^*$ such that $g(X) = \langle \cdot, X \rangle$ for every $X \in H$. Indeed, there is no a priori canonical choice of dual vector we should assign to $g(X)$ as any viable choice need only agree on their application to vectors $Y \in H$; it could very well be the case that $p(Y) = \eta(Y) = \langle Y, X \rangle$ for every $Y \in H$, but $p \neq \eta$. What we can recover from Proposition \ref{prop.posdefisos} is summarized here.

\begin{prop}
Given any inner product $\langle \cdot, \cdot \rangle$ defined on $H$, there exists a unique symmetric, positive semi-definite homomorphism $\beta : T^* \to T$ such that $\beta(T^*) = H$ and $\langle \beta(p), X \rangle = p(X)$ for every $p \in T^*$ and every $X \in H$. Conversely, given any symmetric, positive semi-definite homomorphism $\beta : T^* \to T$ with $\beta(T^*) = H$, there is an inner product $\langle \cdot, \cdot \rangle$ defined uniquely on $H$ by the equality $\langle \beta(p), X \rangle = p(X)$ for every $p \in T^*$ and every $X \in H$.

In other words, there is a canonical bijection between the following spaces.
\begin{enumerate}
\item[$\operatorname{IPH}$:]  Inner products on $H$.
\item[$\operatorname{BH}$:] Symmetric, positive semi-definite linear maps $\beta : T^* \to T$ with image $H$.
\end{enumerate}
\end{prop}
\begin{proof}[Proof Outline]
If $\beta \in \operatorname{BH}$, define $\langle \cdot, \cdot \rangle \in \operatorname{IPH}$  by $\langle X, Y \rangle = \eta\big( \beta(p) \big)$ where $\beta(\eta) = X$ and $\beta(p) = Y$. We must confirm that this is well defined, as the choice for $\eta$ and $\phi$ are not unique. Assume that $\beta(\eta) = \beta(\tilde{\eta})$, then using the symmetry of $\beta$, $\eta\big( \beta(p)\big)  = p \big( \beta(\eta) \big) = p \big( \beta(\tilde{\eta}) \big) = \tilde{\eta} \big( \beta(p) \big)$. Hence, if also $\beta(p) = \beta(\tilde{p})$, then
\[
\eta \big( \beta(p)\big) = \tilde{\eta}\big( \beta(p) \big) = p \big( \beta(\tilde{\eta}) \big) = \tilde{p}\big( \beta(\tilde{\eta}) \big) = \tilde{\eta} \big( \beta(\tilde{p}) \big)
\]
from which we conclude that $\langle \cdot, \cdot \rangle$ is well defined. The remaining pieces to confirm that $\langle \cdot, \cdot \rangle$ is an inner product (on $H$) can be readily checked.

Conversely, if $\langle \cdot, \cdot \rangle \in  \operatorname{IPH}$, let $\mathcal{X} := \{X_1, X_2, ..., X_m\}$ be an basis of $H$ which is orthonormal with respect to $\langle \cdot, \cdot \rangle$; extend this to a basis of $T$, say $\{X_1, X_2, .., X_m, Y_{m+1}, ..., Y_d\},$ and let $\{p^1, p^2, ..., p^m, \eta^{m+1}, ..., \eta^{d}\} \subset T^*$ be the corresponding dual basis. For $p = \sum\limits_{i=1}^m a_i p^i + \sum\limits_{j=m+1}^d b_j \eta^j$, define $\beta(p) := \sum\limits_{i=1}^m a_i X_i \in H$. If we can show that this choice of $\beta$ is well defined, it is then a simple matter to confirm that $\langle \beta(p), X \rangle = p(X)$ for every $p \in T^*$ and $X \in H$,  $\eta \big( \beta(p)\big) = p \big(\beta(\eta)\big)$, and $p\big( \beta(p) \big) \geqslant 0$ for every $p \in T^*$. To ensure that $\beta$ is well defined, suppose that we extend $\mathcal{X}$ to a basis for $T$ as $\{X_1, ..., X_m, \tilde{Y}_{m+1}, ..., \tilde{Y}_{d}\}$ resulting in a corresponding dual basis $\{\tilde{p}_1, ..., \tilde{p}_m, \tilde{\eta}^{m+1}, ..., \tilde{\eta}^d\}$. We need to show that $\beta(p^i) = \beta(\tilde{p}^i)$ for every $i$. It suffices then to show that if $p = \sum a_i p^i + \sum b_j \eta^{j}$ then $p = \sum a_i \tilde{p}^i + \sum c_j \tilde{\eta}^j$; however, this is obvious upon considering $p(X_i)$ for each $X_i \in \mathcal{X}$.
\end{proof}

We conclude this section with one final result that is linear algebraic in nature, but from which the geometric version  Proposition \ref{prop.distinguishedV} follows immediately.
\begin{prop} \label{prop.distinguishedVlinear}
Let $V \subset T$ be a subspace such that $T = H \oplus V$. There exists a unique symmetric, positive semi-definite linear map $g^V : T \to T^*$ such that $\beta \circ g^V(X) = X$ for every $X \in H$, and $g^V(Y) = 0$ for every $Y \in V$. Moreover, if $( \cdot, \cdot)$ is any inner product extending $\langle \cdot, \cdot \rangle$ such that $V = H^{\perp}$, then $g^V(X) = g(X)$ for every $X \in H$, where $g : T \to T^*$ is the isomorphism induced by $( \cdot, \cdot)$. Further, $g^V = g \circ \beta \circ g$.
\end{prop}
\begin{proof}[Proof Outline] \emph{Construction Method 1}:
Let $\{X_1, ..., X_m, Y_{m+1}, ..., Y_d\}$ be a basis for $T$ such that $\spn\{ X_1, ..., X_m\} = H$ and $\spn\{Y_{m+1}, ..., Y_d\} = V.$ From here let $\{ p^1, ..., p^m, \eta^{m+1}, ..., \eta^d \}\subset T^*$ be the dual basis. Then $\spn\{\eta^{m+1}, ..., \eta^{d}\} = \Null(\beta)$ and the restriction of the map $\beta : \spn\{p^1, ..., p^d\} \to H$ is an isomorphism. We denote by $(\beta_V)^{-1} : H \to \spn\{p^1, ..., p^d\}$ the inverse of this isomorphism. Define $g^V = (\beta_V)^{-1} \oplus {\bf 0} : H \oplus V \to T^*$. Note that if $\{\tilde{X}_1, ..., \tilde{X}_m, \tilde{Y}_{m+1}, ..., \tilde{Y}_d \}$ is another basis for $T$ respecting the sum $H \oplus V$ (i.e., $\spn\{\tilde{X_i} : 1 \leq i \leq m\} = H$ and $\spn\{\tilde{Y}_j : m+1 \leq j \leq d\} = V$), then the dual basis $\{\tilde{p}^1, ..., \tilde{p}^m, \tilde{\eta}^{m+1}, ..., \tilde{\eta}^d\}$ satisfies $\spn\{\tilde{p}^1, ..., \tilde{p}^m\} = \spn\{p^1, ..., p^m\}$ and $\spn\{\tilde{\eta}^{m+1}, ..., \tilde{\eta}^{d}\} = \spn\{\eta^{m+1}, ..., \eta^d\}$. From this we can deduce that the choice of $g^V$ really only depends on $V$ and not on the choice of basis of $T$ which respects the sum $H \oplus V$.

\emph{Construction Method 2}: Let $( \cdot, \cdot )$ be any inner product on $T$ which extends $\langle \cdot, \cdot \rangle$ in such a way that $V = H^{\perp}$ with respect to $( \cdot, \cdot )$ (such an extension always exists). Denote by $g$ the isomorphism $T \to T^*$ defined by $g(X) = ( \cdot, X) \in T^*$ for every $X \in T$. Let's note that $g(V) \subset \Null(\beta)$; indeed, if $Y \in V$ and $\eta = g(Y)$, then $\langle \beta(\eta), X \rangle = \eta(X) = (X, Y) = 0$ for every $X \in H$, showing that $\eta \in \Null(\beta)$. In fact, $g(V) = \Null(\beta)$, which is clear once we deduce that $\beta \circ g (X) = X$ for every $X \in H$. To this end, if $X \in H$ and $g(X) = p$, then $\langle \beta(p), Y \rangle = p(Y) = (Y, X) = \langle X, Y \rangle$ for every $Y \in H$ since $( \cdot, \cdot )$ extends $\langle \cdot, \cdot \rangle$; from this it is clear that $\beta(p) = \beta \circ g(X) = X$.  Define $g^V = g \circ \beta \circ g$. Then $g^V (Y) = g \circ (\beta \circ g(Y)) = g(0) = 0$ for every $Y \in V$, and $\beta \circ g^V(X) = \beta \circ g( \beta \circ g(X)) = \beta \circ g(X)  = X$ for every $X \in H$.

\emph{Uniqueness}: Using the notation above, it must be that $g^V = (\beta_V)^{-1} \oplus {\bf 0} : H \oplus V \to T^*$, from which uniqueness follows.
\end{proof}

\begin{acknowledgement}
The authors would  like to thank Alexander Teplyaev for his useful and insightful comments.
\end{acknowledgement}

\end{document}